\documentclass[preprint,12pt]{elsarticle}
\usepackage{amssymb,amsmath,amsthm, comment,bm,lineno}
\usepackage{hyperref}

\numberwithin{equation}{section}

\theoremstyle{plain}
\newtheorem{theorem}{Theorem}[section]
\newtheorem{lemma}[theorem]{Lemma}

\newtheorem{corollary}[theorem]{Corollary}

\def\mod#1{{\ifmmode\text{\rm\ (mod~$#1$)}
\else\discretionary{}{}{\hbox{ }}\rm(mod~$#1$)\fi}}

\theoremstyle{definition}

\newtheorem{example}{Example}[section]
\theoremstyle{remark}
\newtheorem{remark}{Remark}[section]

\newcommand{\al}{\alpha}
\newcommand{\be}{\beta}

\newcommand{\ze}{\zeta}

\newcommand{\la}{\lambda}

\journal{Linear Algebra and its Applications}

\begin{document}

\begin{frontmatter}
\title{On the Waring Rank of Binary Forms}
\author{Neriman Tokcan}
\address{Department of Mathematics, University of 
Illinois at Urbana-Champaign, Urbana, IL 61801} 
\ead{tokcan2@illinois.edu}
\ead[url]{http://rtsl-edge.cs.illinois.edu/nerimanno/}
\date{October 31, 2016}

\begin{abstract}
The $K$-rank of a binary form $f$ in $K[x,y],~K\subseteq \mathbb{C},$  is the smallest number of $d$-th powers of linear forms over $K$ of which $f$ is a $K$-linear combination. We provide  lower bounds for the $\mathbb{C}$-rank (Waring rank)  and for the $\mathbb{R}$-rank (real Waring rank) of binary forms depending on their factorization. We study binary forms with unique $\mathbb{C}$-minimal representation.
\end{abstract}


\end{frontmatter}

\section{Introduction}\label{1}

The main results of this work concern symmetric tensor decomposition which is also known as the Waring problem for forms. Tensors have a rich history and they have recently become ubiquitous in signal processing, statistics, data mining and machine learning~\cite{CG,KB,LT1}.

Let $K[x,y]_d$ denote the vector space of binary forms of degree $d$ with coefficients in the field $K \subseteq \mathbb{C}.$ Given a binary form $f \in K[x,y]_d,$  the $K$-rank of $f,~ L_K(f),$ is the smallest $r$ for which there exist $ \la_j, \al_j, \be_j \in K$ such that  
\begin{equation}\label{E:basic}
f(x,y) = \sum_{j=1}^r \la_j\bigl(\al_{j}x + \be_j y\bigr)^d.
\end{equation}
 A  representation such as \eqref{E:basic}  is
{\it   honest} if the summands are pairwise distinct; that is, if
$\la_i\la_j(\al_i\be_j-\al_j\be_i) \neq 0$ whenever $i \neq j$. Any
representation in which $r = L_K(f)$ is necessarily honest.
In case $K=\mathbb{C}$ or $\mathbb{R},$ the $K$-rank is  commonly called the  Waring rank or  the real Waring rank. Sylvester \cite{S1,S2} presented an algorithm to compute $L_\mathbb{C}(f)$ in 1851 and gave a lower bound for the real Waring rank in 1864. The Waring rank of binary forms has been studied extensively \cite{AH,BBS,CCG,CS,K2,LT,Re1}.  Recently the real Waring rank of binary forms has been investigated \cite{BS,BCG,CKOV,CR,CO}. The relative Waring rank of binary forms over some intermediate fields of  $\mathbb{C} / \mathbb{Q}$ was analyzed in \cite{ Re1,RN}.

It has been known for a long time that $L_\mathbb{C}(f) \leq deg(f).$ This still holds when the underlying field varies, that is, $f \in K[x,y]$ implies $L_K(f) \leq deg(f)$ for any $K \subseteq \mathbb{C}$ \cite[Theorem 4.10]{Re1}. The relation between the number of real roots and the real Waring rank of binary forms has also received substantial attention. Extending the work of Sylvester, Reznick showed that  if $f(x,y)$ is a binary form of degree $d$, not a $d$-th power, with $\tau$ real roots (counting multiplicities), then $L_\mathbb{R}(f) \geq \tau$ \cite[Theorem 3.2]{Re1};  if $f$ is hyperbolic, that is $\tau=d,$ then $L_\mathbb{R}(f)=d$ \cite[Corollary 4.11]{Re1}. The converse was conjectured and proved for $d \leq 4$ in \cite{Re1}. Causa and Re \cite{CR} and Comon and Ottaviani \cite{CO}  showed that the conjecture holds for any square-free binary form, and recently  Blekherman and Sinn \cite{BS} proved that the conjecture is true for any binary form. Reznick presented a classification of binary forms of Waring rank 1 and 2 \cite{Re1}. 

In this paper,  we provide a lower bound for the Waring rank of binary forms based on their factorization over $\mathbb{C}$ (Theorem \ref{complex_bound}). This result also improves the above-mentioned lower bound for  the real Waring rank of binary forms (Corollary \ref{real_bound}). Additionally, we give the classification of binary forms with unique minimal representation (Theorem \ref{classification}) and provide supporting examples Section \ref{5}.

We now outline the remainder of the paper.

In Section \ref{2}, after briefly  discussing apolarity, we recall Sylvester's 1851 Theorem (Theorem \ref{Sylvester}) and the Apolarity Lemma (Theorem \ref{apolarity}). We review the result that  if $f \in K[x,y]_d$ and $k  <\frac{d+2}{2}$ such that $(f^{\perp})_{k-1}=\{0\},$  then $(f^{\perp})_k$ is generated by a projectively unique form  in $K[x,y]$ (Corollary \ref{unique_apolar}). Let $f$ be a binary form of degree $d$ and $L_K(f)=d$, we say that  $f$ has \textit{full $K$-rank}. We include a well-known result on the  binary forms of full Waring rank (Theorem \ref{full_rank}) and a recent result on the real case  (Theorem \ref{full_real_rank}). We conclude this section by recalling the classification of binary forms of Waring rank 1 and 2 which was given in \cite{Re1}. We apply these theorems and observations in Section \ref{3} and Section \ref{4}.

In Section \ref{3}, we first show that if $f$ is a binary form of degree $d,$ not a $d$-th power, and $\alpha_i$ is  a root of  multiplicity $m_i $ of $f,$ then $L_\mathbb{C}(f) \geq m_i +1$ (Theorem \ref{complex_bound}). It directly follows that $L_{\mathbb{C}}({\ell_0}^{d-2}\ell_1\ell_2)=L_\mathbb{R}({\ell_0}^{d-2} p)=d-1$ where $\ell_i$'s are distinct binary linear forms and $p$ is an irreducible quadratic (Corollaries \ref{complex} and \ref{real}). Theorem \ref{complex_bound} combines with \cite[Theorem 3.2]{Re1} into Corollary \ref{real_bound}: if $f$ is a real binary form of degree $d,$ not a $d$-th power, with $\tau$ real roots (counting multiplicities), and $\alpha_i$ is a root of multiplicity $m_i$ of $f,$  then $L_{\mathbb{R}}(f) \geq \max (\tau, m_i +1).$ We then  show that if  $f_\lambda(x,y)=x^{2k}+\binom{2k}{k}\lambda x^k y^k+y^{2k},~ \lambda \neq 0 \in \mathbb{R},~k\geq 3,$ then $L_\mathbb{R}(f_\lambda) \in \{2k-2,2k-1\}$ (Theorem \ref{definite_rank}). The minimal representations of $f_\lambda$ are parameterized in Theorem \ref{definite_rep}. 

In Section \ref{4}, we study binary forms with unique minimal representation. We show that if $f \in K[x,y]_d$ and $L_{\mathbb{C}}(f)=r < \frac{d+2}{2}$, then there exist a field extension $S/K$ such that $L_S(f)=r$ and $[S:K]$ divides $r!$ (Theorem \ref{classification}). We then look at the special case when the underlying field $K$ is a real closed field and $r=3$ (Corollary \ref{real_closed}). 

In Section \ref{5}, we  give examples of binary forms of Waring rank 3 by considering  different cases corresponding to field extensions given in Theorem \ref{classification} and an additional example of a binary quartic form with infinitely many minimal representations of length 3 (Example~\ref{quartic}).

We thank Bruce Reznick for his helpful directions that played a crucial role for the completion of this work. This work will constitute a portion of the author's doctoral dissertation.


\section{Tools and Background}\label{2}
Suppose $p(x,y)= \sum\limits_{i=1}^{d} a_i x^{d-i}y^{i} \in \mathbb{C}[x,y]_d.$ The differential operator associated to $p$ is given by 
\begin{equation*} 
p(D)= \sum\limits_{i=1}^{d} a_i \frac{\partial^d}{\partial x^{d-i}  \partial y ^{i}}~.
\end{equation*}
The \textit {apolar ideal} of $p$, which is denoted by $p^{\perp}$, is the set of all binary forms whose differential operator kills $p$, that is,
\begin{equation*}
p^{\perp} =\{ h \in \mathbb{C}[x,y]~|~ h(D) p = 0 \}. 
\end{equation*} 
This is a homogeneous ideal with the decomposition
\begin{eqnarray*}
p^{\perp}&= & \bigoplus_{ k \geq 0} (p^{\perp} ) _k, \\
(p^{\perp})_k&=&\{ h \in \mathbb{C}[x,y]_k~|~ h(D) p = 0 \}. 
\end{eqnarray*} 

The following theorem is proved in \cite{Re1}, and for $K=\mathbb{C}$, is due to Sylvester \cite{S1,S2} in 1851.
\begin{theorem}\label{Sylvester} \cite[Theorem 2.1,Corollary 2.2]{Re1} Suppose $K\subseteq \mathbb{C}$ is a field,

\begin{equation}
p(x,y)= \sum\limits_{i=1}^{d}\binom di a_i x^{d-i} y^{i}  \in K[x,y]_d
\end{equation}
and  suppose $r \leq d, \alpha_j, \beta_j \in K $ and 

\begin{equation}
h(x,y)= \sum\limits_{t=0}^{r} c_t x^{r-t} y^{t} = \prod \limits _ {j=1}^{r} ( -\beta_j x + \alpha_j y ) 
\end{equation}
is a product of pairwise distinct linear factors. Then there exist $\lambda_k \in K$ such that 

\begin{equation}
p(x,y) = \sum\limits_{k=1}^{r}  \lambda_k ( \alpha_k x + \beta_k y ) ^d
\end{equation}

\noindent if and only if 

\noindent   \begin{equation} \label{linear_system}\left (\begin{array}{cccc}
a_0 &  a_1 & \ldots & a_r  \\
a_1 & a_2 & \ldots & a_{r+1} \\
\vdots & \vdots & \ddots & \vdots\\
a_{d-r} & \ a_{ d- r +1} & \ldots & a_d  \end{array} \right) 
\left( \begin{array}{c}
c_0 \\
c_1 \\
\vdots\\
c_r
 \end{array} \right)=\left( \begin{array}{c}
0 \\
0 \\
\vdots \\
0 \end{array} \right); \end{equation} 

\noindent that is, if and only if 

\begin{equation}
 h(D) p= 0.
\end{equation}
\end{theorem}
If (\ref{linear_system}) holds and $h$ is  square-free, then we say that $h$ is a \textit{Sylvester form} of degree $r$ for $p$. 
Note that a Sylvester form is necessarily separable.

\begin{remark} Let $E_p$ denote the field generated by the coefficients of $p$ over $\mathbb{C}.$ If $h \in K[x,y]$ is a Sylvester form for $p,$ then $E_{p} \subseteq K.$ 
\end{remark}
It is known that any bivariate apolar ideal is a complete intersection ideal and the converse also holds.
\begin{theorem}\cite[Theorem 1.44(iv)]{IK}\label{apolarity}
Let $ p(x,y) \in \mathbb{C}[x,y]_d$. Then $p^{\perp}$ is a complete intersection ideal over $\mathbb{C},$ i.e. $p^{\perp}= \langle f,g \rangle$ such that $deg(f) + deg(g) = d+2$ and $V_\mathbb{C}(f,g) =\emptyset.$ Also, any two such binary forms $f,g$ generate an ideal $p^{\perp}$  for a binary form $p$  of degree $deg(f) + deg(g) -2 $.
\end{theorem}
\begin{corollary}\label{unique_apolar}
Let $p(x,y)$ be a nonzero binary form in $K[x,y]_d$, not a $d$-th power, and suppose that $k < \frac{d+2}{2}$ is the smallest number such that  $(p^{ \perp})_k \neq \{0\}.$ Then there exists a projectively unique binary form $h(x,y) \in K[x,y]_k$ such that $(p^{ \perp})_k = \langle h \rangle.$ Thus, $p(x,y)$ has at most one minimal representation of length $k.$
\end{corollary}

\begin{proof}
We first prove uniqueness:
If $g(x,y)$ is a binary form which is apolar to $p$ and non-proportional  to $h,$ then $deg (g) > k$ by Theorem \ref{apolarity}. It follows that $(p^{\perp})_k$ has a  unique element (up to a scalar multiple). 

We now prove that $h \in K[x,y]_k$: If we take $r=k$, then  the linear system in (\ref{linear_system}) has at least one nonzero solution over $\mathbb{C},$ since $h(x,y)$ corresponds to a solution. Thus, it must have a solution over $K$ as well and by uniqueness  $h(x,y) \in K[x,y]_k.$ 
\end{proof}

The following theorem gives all the binary forms of degree $d\geq 3$ with Waring rank $d.$ It is well known that $L_\mathbb{C}(x^{d-1}y)= d;$ however, to the best of our knowledge, the converse has been proven only later   \cite[Corollary 3]{BBS} and \cite[Ex.11.35]{H}.
\begin{theorem}\label{full_rank}If $d \geq 3$, then $L_\mathbb{C}(f) = d$ if and only if there are two distinct linear forms $\ell_0$ and $\ell_1$ so that $f={\ell_0}^{d-1} \ell_1.$
\end{theorem}
\begin{theorem} \label{full_real_rank}\cite[Theorem 2.2]{BS} Let $f(x,y) \in \mathbb{R}[x,y]_d $ be a binary form of degree $d\geq 3$ and suppose that $f$ is not a $d$-th power. The real Waring rank of $f$ is d if and only if $f$ is hyperbolic.
\end{theorem}

The next tool is  an application of Descartes' Rule of Signs \cite[Question 49, pp.43] {PS}. 

\begin{theorem}\label{Descartes}
Let $a_0 \neq 0, a_n \neq 0,$ and assume that $2m$ consecutive coefficients of the polynomial $a_0 + a_1 x + \ldots + a_n x^n $ vanish, where $m$ is an integer, $m \geq 1.$ Then the polynomial has at least $2m$ non-real zeros.
\end{theorem}

We shall need the following result on the splitting fields \cite[Ex.3, p.30]{K}. Let $[L:K]$ denote the degree of the field extension $L/K.$
\begin{theorem}\label{splitting}
Let $f$ be a polynomial of degree $d$ with coefficients in $K$. Let $L$ be the splitting field of $f$ over $K.$ Then $[L:K]$  is a divisor of $d!.$
\end{theorem}
Binary forms with Waring rank 1 and Waring rank  2 were studied by  Reznick \cite{Re1}.
\begin{theorem} \cite[Theorem 4.1]{Re1} If $p(x,y) \in K[x,y]$, then $L_K (p) = 1 $ if and only if $L_\mathbb{C}(p) =1$.
\end{theorem}
\begin{theorem}\cite[Theorem 4.6]{Re1} Let  $p(x,y)$ be a nonzero binary form of degree $d \geq 3$, and not a $d$-th power, with $\lambda_i,\alpha_i, \beta_i \in \mathbb{C}$ so that 

\begin{equation} \label{rank2}
 p(x,y) =\lambda_1 (\alpha_1 x + \beta _1 y)^d +\lambda_2 ( \alpha_2 x + \beta_2 y )^d \in K[x,y].
\end{equation}

\noindent If (\ref{rank2}) is honest and $L_K(p) >2 $, then there exists $u \in K$ with $\sqrt{u} \not\in K $ so that
 $L_{K(\sqrt{u})}(p)=2 $. The summands in (\ref{rank2}) are conjugates of each other in $K({\sqrt{u}}) $. 
\end{theorem}
\begin{example}
Suppose there exists $\gamma \in \mathbb{Q}~\text{with}~\sqrt{\gamma}  \not\in \mathbb{Q}$ so that
\begin{equation}
p_d(x,y)=\sum\limits_{0\leq 2i \leq d} {d\choose 2i} \gamma^{i}x^{d-2i}y^{2i},~ d \geq 3.
\end{equation}
Then $p_d(x,y)$ is a  rational binary form of Waring rank 2 with the following projectively unique representation: 

\begin{equation}\label{E:rank2}
p_d(x,y)=\frac{1}{2}  ( x + \sqrt{\gamma} y) ^d +\frac{1}{2} ( x- \sqrt{\gamma}y )^d. 
\end{equation}
Notice that the summands in (\ref{E:rank2}) are conjugates of each other in $\mathbb{Q}(\sqrt{\gamma}).$ It follows from Corollary \ref{unique_apolar} that $p$ has a unique representation of length 2, therefore  $L_K(p_d)=2$ if and only if $\sqrt{\gamma} \in K$. 
\end{example}


\begin{section} {A Lower Bound for the Rank of Binary Forms}\label{3}

In this section we give a lower bound for the Waring rank of binary forms based on their factorization over $\mathbb{C}.$ We also improve the lower bound for the real Waring rank of binary forms. 
\end{section}

\begin{theorem}\label{complex_bound} Let  $f(x,y)$ be a nonzero binary form of degree $d$ with the factorization

\begin{equation}
f(x,y)=\prod\limits_{i=0}^{r }\ell_i(x,y)^{m_i}
\end{equation}
where $ r\geq 1$ and ${\ell_i}$'s are distinct linear forms. Then $L_\mathbb{C}(f)\geq \max(m_0,\ldots,m_r)+1$. 
\end{theorem}

\begin{proof}

Assume without loss of generality that $m_0= max(m_0,\ldots,m_r).$ We use the fact that rank is invariant under invertible linear change of variables. After a linear change of variables we may assume $\ell_{0} =y$, then we have 
\begin{equation}
\tilde{f}(x,y)= y^{m_0}g(x,y)~\text{such that}~  y \nmid g(x,y).
\end{equation}
The first $m_0$ coefficients of $ \tilde{f} $ are zero, i.e. $a_0= \ldots = a_{m_{0}-1}=0$ and $a_{m_0} \neq 0.$ Note that $deg(\tilde{f}) \geq m_0+1$, so by setting $r=m_0$, (\ref{linear_system}) becomes:
\begin{equation*}
\begin{pmatrix}
0 & 0& \ldots&0&a_{m_0} \\
0 & 0 & \ldots&a_{m_0}&a_{m_{0}+1} \\
\vdots& \vdots & \vdots&\vdots&\vdots
\end{pmatrix}
\cdot
\begin{pmatrix}
c_0\\\vdots \\ c_{m_0}
\end{pmatrix}
=\begin{pmatrix}
0\\\vdots \\ 0
\end{pmatrix}.
\end{equation*}
Hence, $a_{m_0}c_{m_0} =a_{m_0}c_{{m_{0}}-1} + a_{{m_0}+1}c_{m_0}= 0.$ It follows that $c_{{m_{0}}-1}=c_{m_0}=0$ and every apolar form  of degree $m_0$ is divisible by $x^2$ and $L_\mathbb{C}(f) \geq m_{0}+1$ by Theorem \ref{Sylvester}. 
\end{proof}
 
Landsberg and Teitler \cite[Corollary 4.5]{LT} and Boij, Carlini and Geramita \cite{BCG} have both shown that $L_
\mathbb{C}(x^a y^b) = \max(a + 1, b + 1)$ if $a, b \geq 1.$
\begin{corollary}\label{complex}
Let $f(x,y)=\ell_0(x,y)^{d-2}\ell_1(x,y)\ell_2(x,y)$ such that $d \geq 3$ and $\ell_i$'s are distinct binary linear forms. Then $L_\mathbb{C}(f)=d-1.$
\end{corollary}
\begin{proof}
It follows from Theorem \ref{complex_bound} that $ d-1 \leq L_\mathbb{C}(f)$ and $L_\mathbb{C}(f) \leq d-1$ by Theorem~\ref{full_rank}. Thus, $L_\mathbb{C}(f)=d-1.$
\end{proof}
\begin{corollary}\label{real}
Suppose $f(x,y)=\ell(x,y)^{d-2} p(x,y)$ is a real binary form of degree $d \geq 3$ where $\ell(x,y)$ is a real  linear form and $p(x,y)$ is an irreducible real quadratic form. Then $L_\mathbb{R}(f)=d-1$.
\end{corollary}
\begin{proof}
The Waring rank of $f$ is $d-1$ by Corollary \ref{complex}; therefore, $d-1 \leq L_\mathbb{R}(f).$  On the other  hand, it follows from Theorem \ref{full_real_rank} that $L_\mathbb{R}(f) \leq d-1.$ 
\end{proof}
We can determine the Waring rank of binary cubics based on their factorization \cite[Theorem 5.2]{Re1}.
\begin{remark} 
For $d=4$, we can determine the Waring rank of binary forms with repeating roots based on their factorization. Assume that $\ell_i$'s are distinct  binary linear forms. The following table follows from Theorem \ref{full_rank}, Theorem \ref{complex_bound} and Corollary \ref{complex}.
\begin{center}

\begin{tabular}{ l | c}
$p(x,y)$ & $L_\mathbb{C}(p(x,y))$   \\ \hline
$\ell_0(x,y)^4$ & 1 \\ 
$\ell_0(x,y)^3\ell_1(x,y)$ & 4  \\
$\ell_0(x,y)^2\ell_1(x,y)^2 $ & 3 \\
$\ell_0(x,y)^2\ell_1(x,y) \ell_2(x,y)$ & 3 \\
$\ell_0(x,y) \ell_1(x,y) \ell_2(x,y) \ell_3(x,y)$ & 2,3 \\
\end{tabular}
\end{center}
Notice that a square-free binary quartic form can have Waring rank 2 or 3; for example, $L_\mathbb{C}(4x^3y+6x^2y^2+4xy^3)=L_\mathbb{C}(x^4+4 x^2y^2+y^4)=3$ and $L_\mathbb{C}(8x^3y+36x^2y^2+36xy^3)=L_\mathbb{C}(x^4+y^4)=2.$
\end{remark}
It can be checked from the above table and Theorem \ref{full_real_rank} that  $L_\mathbb{R}(p(x,y)^2)=3$ where $p(x,y)$ is an irreducible real quadratic. This result is also a consequence of the known fact: $L_\mathbb{C}(p(x,y) ^{k})=L_\mathbb{R}(p(x,y) ^{k})=k+1$  \cite[Corollary 5.6]{Re1}. 
\begin{corollary}\label{real_bound}
 Let $f(x,y)$ be a nonzero real binary form of degree $d$ and not a $d$-th  power with the factorization 
\begin{equation}
f(x,y)=\prod\limits_{i=0}^{r}\ell_i(x,y)^{m_i} \prod\limits_{k=0}^{s}p_k(x,y)^{n_k}
\end{equation}
 where $\ell_i$'s are distinct real binary linear forms and $p_k$'s are distinct irreducible real quadratic forms. Then  $L_\mathbb{R}(f) \geq \max \bigl(\sum\limits_{i=0}^{r} m_i, max(m_0,\ldots,m_r,n_0,\ldots,n_s)+1\bigr).$
\end{corollary}

\begin{proof}
The result follows from \cite[Theorem 3.2]{Re1} and  Theorem \ref{complex_bound}.
\end{proof}
Let $f_{\lambda}(x,y)= x^{2k} +\binom{2k}{k} \lambda x^k y^k+y^{2k},~ \lambda\neq 0 \in \mathbb{R}.$ If $|\lambda| \binom{2k}{k}<2,$ then $f_{\lambda}$ is a square-free definite form; therefore, Corollary \ref{real_bound}  does not suggest a lower bound for the real Waring rank of $f_\lambda.$ In the following  theorem, arguments  employing Descartes' Rule of Signs provide a lower bound for $L_\mathbb{R}(f_{\lambda}).$ 

\begin{theorem}\label{definite_rank}
Let $f_{\lambda}(x,y)=x^{2k}+\binom{2k}{k}\lambda x^ky^k+y^{2k}~\text{where}~\lambda\neq 0 \in \mathbb{R}~\text{and}~k \geq 3.$ Then $L_\mathbb{R}(f_{\lambda})\in \{2k-2,2k-1\}.$
\end{theorem}

\begin{proof} First notice that if $k\geq 3,$ then $f_{\lambda}(x,y)$ is not hyperbolic by Theorem \ref{Descartes}. Thus, it follows from Theorem \ref{full_real_rank} that $L_\mathbb{R}(f_{\lambda}) \leq 2k-1.$ 
We then let $r=k +j,$ $0\leq j \leq k-1$ and look for a Sylvester form of degree $r.$ If $k=4,~j=1$, then (\ref{linear_system}) becomes: 
\begin{equation*} \left (\begin{array}{cccccc}
1 & 0 & 0 &0 & \lambda &  0 \\
0 & 0 & 0 & \lambda   & 0&0\\
0&0 &\lambda  & 0 & 0&0\\
0 & \lambda & 0 & 0 &0 &1 \end{array} \right) 
\left( \begin{array}{c}
c_0 \\
c_1 \\
\vdots\\
c_{5}
 \end{array} \right)=\left( \begin{array}{cccc}
0 \\
0\\
0\\
0 \end{array} \right) \Longrightarrow (c_0,c_1,c_2,c_3,c_4,c_5)=(-\lambda c_4, c_1, 0,0,c_4, -\lambda c_1).\end{equation*}
Instead if $k=5,~  j=2$, then $(c_0,c_1,c_2,c_3,c_4,c_5,c_6,c_7)=(-\lambda c_5, c_1, c_2,0,0,c_5, c_6,-\lambda c_2)$ is the solution of the linear system:
\begin{equation*} \left (\begin{array}{cccccccc}
1 & 0 & 0 &0 & 0 & \lambda &  0&0 \\
0 & 0 & 0 & 0 & \lambda & 0 &0&0\\
0&0 & 0& \lambda  & 0 & 0&0&0\\
0 & 0& \lambda & 0 & 0 &0 &0&1 \end{array} \right) 
\left( \begin{array}{c}
c_0 \\
c_1 \\
\vdots\\
c_{7}
 \end{array} \right)=\left( \begin{array}{c}
0 \\
0 \\
0 \\
0 \end{array} \right).
\end{equation*}
In general for $r= k+j,$ we can see that  if $( c_0,c_1,\ldots, c_{k+j} )$ is a solution for  (\ref{linear_system}), then
\begin{eqnarray*}
\begin{aligned}
&c_{i} =0,\hspace{0.1in} j+1 \leq i \leq k-1,\\
&c_0=-\lambda c_k, ~c_{k+j}=-\lambda c_j.
\end{aligned}
\end{eqnarray*}
Therefore, $h_{k+j}(x,y),$ the corresponding Sylvester form of degree $k+j,$ has at least $k-j-1$ consecutive  missing coefficients. If  $h_{k+j}$ splits over $\mathbb{R}$, then $k-j \leq 2$ by Theorem \ref{Descartes}; thus, $2k-2 \leq L_\mathbb{R}(f_{\lambda}).$ 
\end{proof}

The following theorem gives a parametrization for a $\mathbb{C}$-minimal representation of $f_\lambda(x,y)$ as $\lambda$ varies over all nonzero complex numbers. Let $\ze_d$ denote a primitive $d$-th root of unity. 

\begin{theorem}\label{definite_rep}Suppose $f_{\lambda}(x,y)= x^{2k} +\lambda \binom{2k}{k}  x^k y^k+ y^{2k},~\lambda \neq 0.$ Then $L_\mathbb{C}(f_{\lambda})=k$ if $\lambda =\pm 1$ and $k+1$ otherwise. The following is a  minimal representation of $f_{\lambda},$ which is unique for $\lambda =\pm1,$ 
\begin{equation}\label{E:definite}
x^{2k} + \tbinom{2k}{k} \lambda x^k y^k + y^{2k}=(1- \lambda^2 ) y^{2k} + \frac{1}{k}\sum\limits_{i=0}^{k-1} ( x + \lambda^{\frac{1}{k}} \zeta_k^{i}y ) ^{2k}. 
\end{equation}
\end{theorem}
\begin{proof}

We first evaluate the right-hand side of (\ref{E:definite}):
\begin{equation}\label{E:definite_rep}
(1- \lambda^2 ) y^{2k}+\frac{1}{k}\sum\limits_{i=0}^{k-1} ( x +\lambda^{\frac{1}{k}} \zeta_{k}^{i} y )^{2k}= (1- \lambda^2 ) y^{2k}+ \frac{1}{k}\sum\limits_{j=0}^{2k} \tbinom{2k}{j}x^{2k-j}y^{j}  \lambda^{\frac{j}{k}}\big(\sum\limits_{i=0}^{k-1}\zeta_{k}^{ij}\big). 
\end{equation}
The sum $\sum\limits_{i=0}^{k-1}\zeta_{k}^{ij}=0$ unless $k~ | ~j$,  in which case it equals to $k.$ The only multiples of  $k$ in the set $\{ j : 0\leq j \leq 2k\}$  are $0,k,2k.$ The right-hand side of (\ref{E:definite_rep}) reduces to left-hand side of (\ref{E:definite}). If we let $r=k-1$, then the linear system in  (\ref{linear_system}) has only the trivial solution, so $k \leq L_\mathbb{C}(f_{\lambda})\leq k+1.$

If $\lambda \in \{1,-1\}$ then the first summand in (\ref{E:definite_rep}) is zero, therefore $f_{\lambda}$ has a unique minimal representation which is given by (\ref{E:definite}) and $L_\mathbb{C}(f_{\lambda})=k.$ 

Let $\lambda \neq  \pm 1$ and $r=k,$ then the matrix in (\ref{linear_system}) is nonsingular, so $L_\mathbb{C}(f_{\lambda})=k+1.$ Then the minimal representation given by (\ref{E:definite}) is not necessarily unique.
\end{proof}


\begin{section}{Additional Results}\label{4}
 \begin{theorem}\label{S_rank}
Suppose $h(x,y)$ is a Sylvester form of degree $r$ for $f(x,y).$ If $S$ is a splitting field of $h,$ then $L_S(f) \leq r.$ If furthermore there is no Sylvester form of degree $r-1,$ then $L_S(f)=r.$
\end{theorem}
\begin{proof}
The length of the shortest representation of $f$ over $S$ is $L_S(f).$ If $h$ splits over $S,$ then it follows from Theorem \ref{Sylvester} that there exist $\lambda_k,\alpha_k, \beta_k \in S$ such that 
\begin{equation}\label{E:S_rank}
f(x,y) = \sum\limits_{k=1}^{r}  \lambda_k ( \alpha_k x + \beta_k y ) ^d.
\end{equation}
Therefore, $L_S(f) \leq r.$  If  there is no Sylvester form of degree $r-1,$ then the representation in (\ref{E:S_rank}) is a $\mathbb{C}$-minimal representation and $L_S(f)=L_\mathbb{C}(f)= r.$
\end{proof}
\begin{theorem}\label{classification}
Suppose  $r < \frac{d+2}{2}$ and  $f(x,y) \in K[x,y]_d$ with $L_\mathbb{C}(f)=r.$  Then there exists a field extension  $S/K$ such that $L_{S}(f)=r$ and  $[S:K]$ divides $r!$. 
\end{theorem}

\begin{proof}
Let $h(x,y)$ be a Sylvester form of degree $r$ for $f.$ Then $h(x,y) \in K[x,y]$ by Corollary~\ref{unique_apolar}. There is no Sylvester form of degree $r-1,$ since $L_\mathbb{C}(f)=r.$ If $S$ is a splitting field of $h,$ then $L_{S}(f)=r$ by Theorem \ref{S_rank}. Moreover, it follows from Theorem \ref{splitting} that $[S:K] \mid r!.$
\end{proof}

The following lemma gives a special case that $S=K$ when $r=3.$
\begin{lemma}
Suppose $ d \geq 5$  and   there exist nonzero $\lambda_i, \alpha_1, \beta_1  \in \mathbb{C}$ so that 
\begin{equation}\label{E:S=K}
f(x,y)= \lambda_1(\alpha_1 x + \beta _ 1 y ) ^d +\lambda_2  x^d + \lambda_3 y ^d  \in K [x,y].
\end{equation}
Then $L_K(f)=3$, and (\ref{E:S=K}) is the projectively unique representation of $f$ of length 3. 
\end{lemma} 
 \begin{proof} 
The Sylvester form corresponding to (\ref{E:S=K}) is $h(x,y) = ( \beta_1x- \alpha_1 y ) y x$ by Theorem \ref{Sylvester}. It follows from Corollary \ref{unique_apolar} that $h \in K[x,y];$ thus, $h$ splits over $K$ and $L_K(f)=3.$
\end{proof}

 The following corollary concerns a special case of  Theorem \ref{classification} where  $r=3$ and $K$ is a real closed field. 
\begin{corollary}\label{real_closed}

Suppose $d \geq 5,~ K \subseteq \mathbb{C}$ is real closed field and there exist $\lambda_i, \alpha_i, \beta_i \in \mathbb{C}$ such that 
\begin{equation}\label{E:real_closed}
f(x,y)= \lambda_1(\alpha_1 x + \beta _ 1 y ) ^d + \lambda_2( \alpha_2 x + \beta _ 2 x ) ^d + \lambda_3(\alpha_3 x + \beta_3 y) ^d  \in K [x,y]
\end{equation}
is a honest representation and  $L_K(f) > 3.~\text{If}~u=(\beta_1\alpha_2-\beta_2\alpha_1)^2(\beta_1\alpha_3-\beta_3\alpha_1)^2(\beta_2\alpha_3-~\beta_3\alpha_2)^2,$  then $L _{K (\sqrt{u})}(f) = 3.$ One of the summands in (\ref{E:real_closed}) is in $K[x,y]$  whereas the other two summands are conjugates of each other in $K(\sqrt{u}).$

\end{corollary}

\begin{proof}
It follows from Theorem \ref{Sylvester} and Corollary \ref{unique_apolar} that the projectively unique Sylvester form of degree 3 for $f$ is given by
\begin{equation*}
h(x,y)=(\beta_1 x - \alpha_1y)(\beta_2 x - \alpha_2y)(\beta_3 x - \alpha_3y) \in K[x,y].
\end{equation*}
 Notice that $u$ equals the discriminant of $h(x,1).$ By the hypothesis $h$ does not split over $K;$ therefore, $\sqrt{u} \not\in K.$ Since $h$ is an odd degree form over a real closed field, it must have a factor  in $K[x,y]$ (see \cite[Theorem 1.2.2]{BC}). Thus, only one factor of $h(x,y)$ is in $K[x,y]$ and the other two are conjugates of each other in $K(\sqrt{u}).$
Note that every field automorphism which fixes $K$ permutes the summands in (\ref{E:real_closed}). If we consider the conjugation with respect to $\sqrt{u}$, then (\ref{E:real_closed}) has two summands which are conjugates of each other in $K(\sqrt{u})$ and a summand in $K[x,y]$. 
 \end{proof}


\end{section}


\begin{section}{Examples}\label{5}

Let $f$ be a binary form of degree $d\geq 5$ in $K[x,y]$ and $L_\mathbb{C}(f)=3.$ Then it follows from Theorem \ref{classification} that there exist a field extension $S/K$ with $L_K(f)=3.$ Either $S=K$ or else $S/K$ has degree 2,3, or 6. 
In this section, we give examples for each case of $S/K$  and additional one showing that if $d\leq 4$, there can be infinitely many  representations of length~3.

\begin{example} This example displays a form falling into the case where $S/K~\text{has degree 2}$.
\begin{equation*} 
\text{Let}~f(x,y)= (1+2 \sqrt{2}) x^5-25 x^4y+(60 \sqrt{2} +10) x^3y^2-170 x^2y^3+(90\sqrt{2} +5)x y^4~-53y^5
\end{equation*}
First, with $r=2$, we see that the matrix from (\ref{linear_system}) is nonsingular, hence $L_\mathbb{C}(f) \geq 3.$ On taking $r=3$, we get the Sylvester form:
\begin{equation*}
h(x,y) =3x^3-3x^2y - x y^2 + y^3= ( y - \sqrt{3}x) ( y + \sqrt{3} x )( y - x).
\end{equation*}
 Thus, we arrive the following conclusion: $L_S(f)=3~\text{if and only if}~\mathbb{Q}(\sqrt{2},\sqrt{3}) \subseteq S$ with the corresponding representation:
\begin{equation*}
f(x,y)=(\sqrt{2} + \sqrt{3}) ( x - \sqrt{3}y ) ^5 + (\sqrt{2} - \sqrt{3})( x + \sqrt{3} y ) ^5 + ( x+y)^5.
\end{equation*}
Notice that the above representation has two summands which are conjugates of each other under the conjugation with respect to $\sqrt{3}$ in  $\mathbb{Q}(\sqrt{2})$ and a summand in $\mathbb{Q}(\sqrt{2})[x,y]$.\

In particular, if $S=\mathbb{Q}(\sqrt{2},\sqrt{3}),$ then $[S:\mathbb{Q}(\sqrt{2})]=2$ and $L_S(f)=3.$
\end{example}
\begin{example}
\noindent Let $f(x,y)= -15 x^5 + 90 x^4y -30 x^3 y^2+ 60 x^2 y^3+ 3y^5.$  If we set $r=2$, then the solution to the linear system in (\ref{linear_system}) is trivial, so $L_\mathbb{C}(f) \geq 3.$  

If we set $r=3$ in (\ref{linear_system}), then up to a scalar multiple $h(x,y)= x^3- 3x y^2 + y^3.$ We can factorize $h(x,y)$ by using the trigonometric identity $4 \cos^3(\theta) - 3 \cos(\theta) = \cos(3\theta)$:
\begin{equation*}
h(x,y)= (x-  2 \cos{\tfrac{2\pi}{9}}y)(x- 2 \cos{\tfrac{4\pi}{9}} y)(x- 2 \cos{\tfrac{8\pi}{9}} y).
\end{equation*}
Therefore, $L_S(f)=3~\text{if and only if}~\mathbb{Q}(\cos{\tfrac{2\pi }{9}}) \subseteq S$ with the $\mathbb{C}$-minimal representation:
\begin{equation*}
f(x,y)= ( y +  2 x \cos{\tfrac{2\pi}{9}})^5 + ( y+  2 x \cos{\tfrac{4\pi}{9}})^5 +  (y+ 2 x \cos{\tfrac{8\pi}{9}})^5.
\end{equation*}
If we let $S=\mathbb{Q}(\cos{\tfrac{2\pi }{9}}),$ then we have $[S:\mathbb{Q}]=3$ and $L_S(f)=3.$
\end{example}

\begin{example}Let $f(x,y)= 3 x^7 + 210 x^4 y^3 + 84x y^6.$ Then, $(f^{\perp})_2$ is empty  by Theorem~\ref{Sylvester}; thus, the Waring rank of $f$ is at least 3. The Sylvester form of degree 3 is 
\begin{equation*}
h(x,y)=  y^3-2x^3 = ( y -\sqrt[3] {2} x) ( y -\sqrt[3] {2} \omega x) (y-\sqrt[3] {2} \omega^2 x), ~  \omega= e^{\frac{ 2 \pi i}{3}}.
\end{equation*}
 Note that $h$ splits over $\mathbb{Q }(\sqrt[3] {2} ,\sqrt {-3} ).$  The $\mathbb{C}$-minimal representation of $f$ is given by
\begin{equation*}
f(x,y)=( x+ \sqrt[3] {2} y ) ^7+( x+ \sqrt[3] {2} \omega y ) ^7 +( x+ \sqrt[3] {2} \omega^2 y ) ^7. 
\end{equation*}
Let $S=\mathbb{Q }(\sqrt[3] {2} ,\sqrt {-3} ),$ then $[S:\mathbb{Q}]=6$ and $L_S(f)=3.$
\end{example}

If the degree of a binary form is less than 5, then the Sylvester form of degree 3 does not need to be unique. There can be infinitely many representations of length~3.

\begin{example}\label{quartic} Let $f(x,y)= ( x^2 + y^2) ^2,$ then by  \cite[Corollary 5.6] {Re1}  $L_S(f)=3$ if and only if $\mathbb{Q}(\sqrt{3} ) \subseteq  S$ with the minimal representations 
\begin{equation*}
( x^2 + y^2 )^2=\frac{1}{18}\sum\limits_{i=0}^{2}(\cos (\tfrac{i\pi}{3} + \theta)x + \sin (\tfrac{i\pi}{3} + \theta)y ) ^4,~~ \theta \in \mathbb{C}.
\end{equation*}

\end{example}
\end{section} 

\section{Acknowledgment}

We would like to express our gratitude to Zach Teitler. His insightful comments contributed significantly to the improvement of this work.  The author was supported by UIUC Campus Research Board Grant.


\bibliographystyle{elsarticle-num}
  \bibliography{reference}

\end{document}